\documentclass[review]{elsarticle}

\usepackage{lineno,hyperref}
\usepackage{amssymb}
\usepackage{amsmath}
\usepackage{amsthm}

\newtheorem{theorem}{Theorem}
\newtheorem{lemma}[theorem]{Lemma}
\newdefinition{remark}{Remark}

\modulolinenumbers[5]

\journal{European Journal of Combinatorics}
\newcommand\numberthis{\addtocounter{equation}{1}\tag{\theequation}}

\begin{document}

\begin{frontmatter}

\title{Universal Sets and Cover-Free Families}

\author{Debjyoti Saharoy\corref{cor1}}
\ead{debjyoti.saharoy@xerox.com}
\address{Xerox Research Centre India, Bangalore, Karnataka, India 560103}

\author{Shailesh Vaya}
\ead{shailesh.vaya@xerox.com}
\address{Xerox Research Centre India, Bangalore, Karnataka, India 560103}

\cortext[cor1]{Corresponding author}

\begin{abstract}
We propose a polynomial time construction of an $(n,d)$-universal set over alphabet $\Sigma=\{0,1\}$, of size $d\cdot 2^{d+o(d)}\cdot\log n$. This is an improvement over the size, $d^{5}2^{2.66d}\log n$, of an $(n,d)$-universal set constructed by Bshouty, \cite{BshoutyTesters}, over alphabet $\Sigma=\{0,1\}$.
\end{abstract}


\end{frontmatter}

\linenumbers

\section{Introduction}
\label{intro}
  An $(n,d)$-universal set $\mathcal{U}$ over an alphabet $\Sigma$ is a family of vectors, $\mathcal{U}\subseteq\Sigma^{n}$, such that for any index set $S \subset [n]$, with $|S| = d$, the projection of $\mathcal{U}$ on $S$ contains all possible $|\Sigma|^d$ configurations. Universal sets have universal appeal in almost all scientific disciplines which have concerns regarding testing, where the particular coordinates whose combinatorial possibilities are to be tested are hidden from the tester. They have been intensively studied in the name of $t$-coverage arrays and are referred to as universal sets in contemporary combinatorics literature. In this short note, we elucidate a construction of universal sets using cover free families, for $|\Sigma|=2$, which implies polynomial time construction of almost optimal size universal sets for $|\Sigma|=2$. However, our construction does not depend on the construction of cover-free families. 
In the remaining part of this section we formally define an $(n,d)$-universal set and $(n,(r,s))-CFF$. And also state the related results on the size of their constructions. In section \ref{construction} we give the construction of universal sets in Lemma \ref{ConstructionLemma} and prove the size bounds in Theorem \ref{MainTheorem}. In section \ref{Consequences} we point some other consequences of this result. 
\subsection{$(n,d)$-Universal Set}
\label{universalsets}
  An $(n,d)$-universal set over an alphabet $\Sigma$ is a set $\mathcal{U}\subseteq\Sigma^{n}$ such that for every $1\leq i_1<i_2<\cdots<i_d\leq n$ and every $(\sigma_1,\cdots,\sigma_d) \in \Sigma^{d}$ there is \textit{\textbf{a}} $\in \mathcal{U}$ such that $a_{i_j} = \sigma_{j}$ for all $j = 1,\cdots,d.$ If $|\Sigma| = q$, then $U(n,d,q)$ is the size of the smallest $(n,d)$-universal set over the alphabet $\Sigma$. The union bound shows that there is an $(n,d)$-universal set over an alphabet $\Sigma$ of size 
\begin{align*}
U(n,d,q) \leq dq^{d}(ln \frac{n}{d} + ln q) = O(dq^{d}\log n).
\end{align*}
\par Obviously, finding a small $(n,d)$-universal set is a \textit{d-restriction problem} \cite{naor1995splitters}. For $q=2$, a lower bound of $\Omega(2^d\log n)$ was proved in \cite{kleitman1973lowerbound}. The polynomial time (i.e $poly(q^d,n)$) construction for this problem of size $d^{O(\log d/\log q)}q^d\log n$ for $q<d$, \cite{naor1995splitters}, \cite{alon1992construction}, was improved by Bshouty \cite{BshoutyTesters}. Specifically, for $q=2$, \cite{BshoutyTesters} gave a polynomial time construction of $(n,d)$-universal set of size not exceeding $d^5 2^{2.66d}\log n$.  

\subsection{Cover-Free Family}
\label{coverfree}
\par Let us fix positive integers $r,s,n$ with $r,s < n$ and let $d:=r+s$. Let $X$ be a set with $N$ elements and let $\mathcal{B}$ be a set of subsets (blocks) of $X$, $|\mathcal{B}|=n$. Then $(X,\mathcal{B})$ is a $(n,(r,s))$\textit{-cover-free family} $((n,(r,s))-CFF)$, \cite{kautz}, if for any r blocks $B_1,\cdots,B_w \in \mathcal{B}$ and any other $s$ blocks $A_1,\cdots,A_r \in \mathcal{B}$, we have
\begin{align*}
\bigcap\limits_{i=1}^{r} B_{i} \not\subseteq \bigcup\limits_{i=1}^{s} A_{j}.
\end{align*} 
\par Equivalently, given an $(n,(r,s))$-CFF $\mathcal{F}$, denote $N=|\mathcal{F}|$ and construct the $N*n$ boolean matrix $A$ whose rows are the elements of $\mathcal{F}$ and the columns can be thought of as the characteristic vectors of subsets. If $\mathcal{B} = {B_1,\cdots,B_n}$ denotes the set of blocks corresponding to the columns, then A is the incidence matrix of $\mathcal{B}$, i.e. the $i^{th}$ element of $X$ is in $B_j$ iff $A_{i,j} = 1$.
\par The CFF property of $\mathcal{F}$ implies that for any $r$ blocks 
$B_1,\cdots,B_r \in \mathcal{B}$ and any other $s$ blocks $A_1,\cdots,A_s \in \mathcal{B}$ (distinct from the $B$'s), there is an element of $X$ contained in all the $B's$ but not in any of the $A$'s. Let $N(n,(r,s))$ denote the minimum size of any $(n,(r,s))$-CFF. 
\par D'yachkov et. al.'s breakthrough result, \cite{Dyachkov2014}, implies that for $s,n \rightarrow \infty$
\begin{align*}
N(n,(r,s)) = \Theta(N(r,s)\cdot\log n). \numberthis \label{Dyachkov} 
\end{align*} 
where
\begin{align*}
N(r,s) := \frac{d{d\choose r}}{\log {d\choose r}}. 
\end{align*}
\par This bound is non-constructive. Bshouty et. al., \cite{BshoutyAlmostOptCFF} calls an (n,(r,s))-CFF $\mathcal{F}$ \textit{almost optimal}, if
it's size $N=|\mathcal{F}|$ satisfies
\begin{align*}
N = N(r,s)^{1+o(1)}\cdot \log n  
\end{align*}
and for $r=O(d)$
\begin{align*}
N = N(r,s)^{1+o(1)}\cdot \log n = 2^{H_2(r/d)d+o(d)}\cdot \log n \numberthis \label{eqn:almostOptCFF}
\end{align*}
where $H_2(x)$ is the binary entropy function. The term $o(1)$
is independent of $n$ and tends to $0$ as $d \rightarrow \infty$.
A CFF family $\mathcal{F}$ is said to be constructed in linear time
if it can be constructed in time $O(N(r,s)^{1+o(1)}\cdot \log n \cdot n)$.
\par Bshouty \cite{BshoutyTesters,BshoutyLT} and Bshouty et. al. \cite{BshoutyAlmostOptCFF} constructed almost optimal $(n,(r,s))$-CFF $\mathcal{F}$ for $r<d^{o(1)}$ and
$d^{o(1)}<r<\omega(d/(\log\log d\log\log\log d))$ respectively, in linear time. Fomin et. al. \cite{Fomin} constructed 
almost optimal $(n,(r,s))$-CFF $\mathcal{F}$ for $r>\omega(d/(\log\log d\log\log\log d))$ in linear time.

\section{Construction of Universal Sets}
\label{construction}
  We give the explicit (i.e polynomial time) construction of $(n,d,q)$-universal set $\mathcal{U}$, for $q=2$, using $(n,(r,s))$-CFF $\mathcal{F}$. We use the explicit linear time construction of almost optimal size $(n,(r,s))$-CFF $\mathcal{F}$ given in \cite{BshoutyTesters,BshoutyAlmostOptCFF,BshoutyLT,Fomin}. 

\par \noindent \textit{Notation:} Let us denote an $(n,d,q)$-universal set $\mathcal{U}$ by $\mathcal{U}_{(n,d,q)}$. We suppress $q$, if $q=2$. We denote an $(n,(r,s))$-CFF $\mathcal{F}$ by $\mathcal{F}_{(n,(r,s))}$. Our construction is based on the following lemma.
\begin{lemma}\label{ConstructionLemma}
\begin{align*}
\mathcal{U}_{(n,d)} = \bigcup\limits_{i=0}^{d-1} \mathcal{F}_{(n,(i,d-i))} \numberthis \label{2UniversalSet} 
\end{align*}
$\mathcal{U}_{(n,d)}$ is an $(n,d)$-universal set over alphabet $\Sigma=\{0,1\}$.
If $N(n,(d/2,d/2))$ denotes the size of an optimal $(n,(d/2,d/2))$-CFF, then $|\mathcal{U}_{(n,d)}|\leq d\cdot N(n,(d/2,d/2))$. Moreover, $|\mathcal{U}_{n,d}|$ is in asymptotic equivalence with $N(n,(r,s))$ for $r=O(d)$.  
\end{lemma}

\begin{proof}
First we prove that $\mathcal{U}_{(n,d)}$ is indeed an $(n,d)$-universal set. The proof is by contradiction. Let us assume that exists some  $1\leq i_1<i_2<\cdots<i_k\leq n$ and $(\sigma_1,\cdots,\sigma_d)\in \Sigma^d$ such that for no $a\in \mathcal{U}_{(n,d)}$, $a_{i_j}=\sigma_j$ for all $j=1\cdots,d$. Without loss of generality, let us assume that the chosen $(\sigma_1,\cdots,\sigma_d)$ has $r$ $1's$ and $s$ $0's$ where $r+s=d$.
  In other words, the $(\sigma_1,\cdots,\sigma_d)\not\in \mathcal{F}_{(n,(r,s))}$. This is not possible by the definition of $(n,(r,s))$-CFF $\mathcal{F}$. Hence the contradiction.
\par We now calculate the size of the $\mathcal{U}_{(n,d)}$. 
\begin{align*}
|\mathcal{U}_{(n,d)}| &= |\bigcup\limits_{i=0}^{d-1} \mathcal{F}_{(n,(i,d-i))}| \\
&= 2\cdot|\bigcup\limits_{i=0}^{\frac{d}{2}-1} \mathcal{F}_{(n,(i,d-i))}| \\
&\leq d\cdot N(n,(d/2,d/2)) \numberthis \label{UniversalSetBound}
\end{align*} 
\par The second equality follows from the fact that we can consider $r\leq d/2$, because if not, one can construct an $\mathcal{F}_{(n,(s,r))}$ and take the set of complement vectors. The first inequality follows from the fact that the size of $\mathcal{F}_{(n,(r,s))}$ for $r=O(d)$ dominates $r=O(1)$, $r=\omega(1)$ and
$r=o(d)$. It must be noted that the bound on the size of $\mathcal{U}_{(n,d)}$ is in asymptotic equivalence with bound on the size of an $(n,(r,s))$-CFF $\mathcal{F}$ for $r=O(d)$..
We can claim so because any \textit{optimal} or \textit{almost optimal} construction of an $(n,(r,s))$-CFF $\mathcal{F}$ must obey the tight bound given by D'yachkov et. al, \cite{Dyachkov2014}, in eq. $(1)$, and the bound (i.e the quantity $N(n,(r,s))$) in eq. $(1)$ is monotonically increasing in $d$. \\   
\end{proof}

  Bshouty, \cite{BshoutyTesters}, constructed an $(n,d)$-universal set of size $d^{5}2^{2.66d}\log n$ over an alphabet $\Sigma=\{0,1\}$. To the best of our knowledge this is the best polynomial time construction for this problem over an alphabet of size $2$. We give the following theorem which improves this size.
\begin{theorem}\label{MainTheorem}
$\mathcal{U}_{n,d}$ is an explicitly (i.e in polynomial time) constructed $(n,d)$-universal set over alphabet $\Sigma=\{0,1\}$, of size $d\cdot 2^{d+o(d)}\cdot\log n$. 
\end{theorem}
\begin{proof}
Using inequality $(4)$ and almost optimal construction of an $(n,(r,s))$-CFF by \cite{BshoutyTesters,BshoutyAlmostOptCFF,BshoutyLT,Fomin} for $r=O(d)$,
\begin{align*}
|\mathcal{U}_{(n,d)}| &\leq d\cdot N(n,(d/2,d/2)) \\
&= d\cdot N(d/2,d/2)^{1+o(1)}\cdot\log n \\
&= d\cdot 2^{H_{2}(1/2)d+o(d)}\cdot\log n \\
&= d\cdot 2^{d+o(d)}\cdot \log n 
\end{align*}
The first equality follows from eq. $(2)$. The polynomial time taken in the construction follows from the fact that we take the union of $d$, $(n,(d/2, d/2))$-CFF each of which is constructed in linear time.\\
\end{proof}  
\begin{remark}
We also make the observation that the construction of $(n,d,2)$-universal set in Lemma \ref{ConstructionLemma} can also be extended for small $q$'s  greater than $2$ (like $q=3,4$) by adapting Colbourn et. al.'s, \cite{colbourn2006products}, construction of product of covering arrays.  
\end{remark}
  
%
 
\section{Some Consequences of $(n,d)$-Universal Set}
\label{Consequences}
The improved construction of the $(n,d)$-universal set have direct
consequences in the problem of fault-tolerance of an hypercube,
\cite{seroussi1988vector}. It also improves the running time  of Blum and Rudich's, \cite{berger1991simulating}, learning algorithm for $k$-term DNFs as well as Bshouty's, \cite{bshouty1993exact}, learning algorithm for $k$-CNF. As pointed out by Naor in 
\cite{naor1995splitters}, improved construction of universal sets
also improves the non-approximability results of the \textit{set cover} problem. It also finds application in distributed colouring: provides a constructive argument of the
existence of recoloring protocols of Szegedy and Vishwanathan \cite{szegedy1993locality}.

\bibliography{Paper}

\end{document}